\documentclass[12pt]{article}
\usepackage{amssymb}
\usepackage{amsthm}
\usepackage{amsmath}
\usepackage{graphicx}

\newtheorem{theorem}{Theorem}
\newtheorem{lemma}[theorem]{Lemma}

\newtheorem{Lemma A.1}{Lemma A.1}
\theoremstyle{definition}

\theoremstyle{remark}

\def\be{\begin{equation}}
\def\ee{\end{equation}}

\begin{document}
\title{ A Liouville Theorem for Axi-symmetric Navier-Stokes Equations on $\mathbb{R}^2 \times \mathbb{T}^1$ }

\author{Zhen Lei \footnotemark[1]
\and Xiao Ren   \footnotemark[1]
\and Qi S. Zhang    \footnotemark[2]
}
\renewcommand{\thefootnote}{\fnsymbol{footnote}}
\footnotetext[1]{
School of Mathematical Sciences; LMNS and Shanghai Key Laboratory for Contemporary Applied Mathematics,
Fudan University, Shanghai 200433, P. R.China.}
\footnotetext[2]{Department of
Mathematics,  University of California, Riverside, CA 92521, USA.}

\date{\today}

\maketitle

\begin{abstract}
We establish a Liouville theorem for bounded mild ancient solutions to the axi-symmetric incompressible Navier-Stokes equations on $(-\infty, 0] \times (\mathbb{R}^2 \times \mathbb{T}^1)$. This is a step forward to completely solve the conjecture  on $(-\infty, 0] \times \mathbb{R}^3$  which was made in \cite{KNSS} to describe the potential singularity structures of the Cauchy problem.
\end{abstract}

\section{Introduction}
In the analysis of many physical or geometric PDEs, one often applies the standard blow-up procedure to understand the local structures of potential singularities. Such a procedure for parabolic PDEs naturally produces limit solutions with certain \textit{a priori} estimates, which exist on the half space-time domain $(- \infty, 0] \times \mathbb{R}^n$ and are called to be ancient. Liouville properties of these ancient solutions play important roles in the study of singularity structures of PDEs.

In this article, we establish a Liouville theorem for ancient solutions of three-dimensional incompressible Navier-Stokes equations. The system of equations are
\begin{equation} \label{eq-NS}
\begin{cases}
\partial_t\mathbf{v} + (\mathbf{v} \cdot \nabla) \mathbf{v}  + \nabla p = \Delta \mathbf{v}, \\
\text{div}\,\, \mathbf{v} = 0.\\
\end{cases}
\end{equation}
Here $\mathbf{v}$ is the velocity vector and $p$  the pressure. They are the fundamental equations describing the motion of viscous fluid substances and are believed to  describe turbulence properly\cite{LT}. Whether singularities can develop in finite time from smooth initial data has been called  one of the seven most important open problems in mathematics by the Clay Mathematics Institute\cite{F}.

We focus on the axi-symmetric case. In cylindrical coordinates $(r, \theta, z)$ with the basis vectors:
\begin{equation}\nonumber
\textbf{e}_r = \begin{pmatrix}\cos\theta \\ \sin\theta \\ 0\end{pmatrix},\quad \textbf{e}_\theta = \begin{pmatrix}- \sin\theta \\ \cos\theta \\ 0\end{pmatrix},\quad \textbf{e}_z = \begin{pmatrix} 0 \\ 0 \\ 1\end{pmatrix},
\end{equation}
we write $\mathbf{v}=v^r \mathbf{e}_r + v^\theta \mathbf{e}_\theta  + v^z \mathbf{e}_z$. By \textit{axi-symmetric}, we mean that $v^r$, $v^z$ and $v^\theta$ depend only on $(t, r, z)$, but not on $\theta$. Here is the main result, whose proof will be given in subsequent sections.
\begin{theorem}
\label{thzhouqi}
Let $\mathbf{v}=v^r \mathbf{e}_r + v^\theta \mathbf{e}_\theta  + v^z \mathbf{e}_z $ be a bounded ancient mild solution to the axi-symmetric Navier-Stokes equations, such that $\Gamma= r v^\theta$ is bounded. Suppose $\mathbf{v}$ is periodic in the $z$ variable. Then $\mathbf{v}\equiv c \mathbf{e}_z$ where $c$ is a constant.
\end{theorem}

We emphasize that, in \cite{KNSS}, G. Koch, N. Nadirashvili, G. A. Seregin and V. ${\rm \breve{S}}$ver${\rm \acute{a}}$k \cite{KNSS} conjectured that bounded mild ancient solutions of axi-symmetric Navier-Stokes equations are constants. They also partially proved it under the conditions $\|\Gamma\|_{L^\infty} \leq C_\ast$ and $|\textbf{v}| \leq \frac{C_\ast}{r}$ (or $|\textbf{v}| \leq \frac{C_\ast}{\sqrt{-t}}$) instead of the above periodicity. In independent works of  C. C. Chen, R. M. Strain, H. T. Yau and T. P. Tsai \cite{CSTY1, CSTY2}, regularity of solutions under similar conditions as in \cite{KNSS} is proved. See also \cite{LZ11,S2011} for a generalization of these Liouville type and regularity results to the case that $\mathbf{v} \in L^\infty(0, T; BMO^{-1}_x)$. Indeed, while those above conditions are unverified constraints imposed on solutions, they imply the boundedness of certain scale invariant energy quantities, which further implies the assumptions in many known results on regularity of axi-symmetric solutions, see \cite{S2011} for more details. We also mention that global regularity in the case of axial symmetry with zero swirl has been treated in the classical work \cite{L1968}. In the recent work \cite{LZ2017}, the critical nature of axi-symmetric Navier-Stokes equations is exploited and regularity is proved under the condition that $|\Gamma| \le \frac{C_\ast}{|\ln r|^2}$ near the symmetry axis (see \cite{Wei} for an improvement).

Theorem \ref{thzhouqi} settles the important conjecture in \cite{KNSS}, in the case that $\mathbf{v}$ is periodic in $z$. A crucial observation in the proof is to connect the compactness of $\mathbb{T}^1$ and the oscillation of the stream functions (see the proof below for details). The result might help to enhance the understandings of the local structure of potential singularities of the Cauchy problem, since the blowup procedure applied to the potential singularities of solutions of the Cauchy problem may, at least as a possible case, produce bounded ancient solutions which are periodic in $z$. The result is also interesting by itself, as a study on the Liouville property of the bounded mild ancient solutions to the axi-symmetric Navier-Stokes equations.

Note that in the theorem we have imposed the constraints that $\mathbf{v}$ and $\Gamma$ are bounded, and $\mathbf{v}$ is mild. In fact, they are  all very natural as have been explained in \cite{KNSS}. For a self-contained presentation, here we still give a brief explanation.

Denote by $\mathbb{P}$ the Helmholtz projection of vector fields onto divergence free fields, and by $\mathbb{S}$ the solution operator of the heat equation (i.e. convolution with the heat kernel). A solution $\mathbf{v}$ is said to be \textit{mild} on the time interval $[0,T)$, if the following integral version of \eqref{eq-NS}
\begin{equation*}
\mathbf{v}(t) = \mathbb{S}(t-s)\mathbf{v}(s) + \int_s^t \mathbb{S}(t-\tau) \mathbb{P} \text{div}\, \mathbf{\mathbf{v}(\tau) \otimes \mathbf{v}(\tau)} d\tau
\end{equation*}
holds for all $0\le s \le t < T$.

An important feature of the axi-symmetric Navier-Stokes equations is that $\Gamma$ satisfies the following equation
\begin{equation}\label{Gamma/vtheta}
\partial_t\Gamma +(\mathbf{b}\cdot\nabla)\Gamma+\frac{2}{r}\partial_r
\Gamma= \Delta \Gamma,
\end{equation}
where
\begin{equation} \label{eq-b}
\mathbf{b} = v^r \mathbf{e}_r + v^z \mathbf{e}_z, \quad \text{div} \,  \mathbf{b}=0.
\end{equation}
Then the parabolic maximum principle for $\Gamma$ implies that  $\|\Gamma\|_{L^\infty}$ is bounded uniformly in time if it is bounded initially.

Now we consider the mild solution $\mathbf{u}$ to the Cauchy problem of axi-symmetric Navier-Stokes equations starting from smooth initial data with sufficient spatial decay. Suppose that $\mathbf{u}$ is smooth on $[0, T)$ and blows up at some point $(T, x_0)$. In view of the partial regularity theory of Caffarelli-Kohn-Nirenberg \cite{CKN}, one concludes that $x_0$ must be on the axis so that $x_0 = (0, 0, z_0)$, and $z_0$ must be bounded. Moreover, one can find a sequence of points $(t_n, x_n)$ with $t_n \nearrow T$ and $x_n = (x_{1n},0, z_n) \to x_0'$ such that
$$M_n = |\mathbf{u}(t_n, x_n)| = \sup_{t \leq t_n}|\mathbf{u}(t, \cdot)| \nearrow \infty.$$
To study the local singularity structure, we apply the blow-up procedure to $\mathbf{u}$:
$$\mathbf{u}^{(n)}(t, x) = \frac{1}{M_n}\mathbf{u}(t_n + \frac{t}{M_n^2}, x_n + \frac{x}{M_n}).$$
Clearly, $\mathbf{u}^{(n)}$ is a sequence of bounded mild soultions of the Navier-Stokes equations. By \cite{KNSS} and \cite{LZpjm}, $\mathbf{u}^{(n)}$ locally uniformly converges to a bounded mild ancient solution $\mathbf{v}$ to the Navier-Stokes equations which is either a two-dimensional one  or a three-dimensional axi-symmetric one. If $\mathbf{v}$ is two-dimensional, then we fall into a simple case and $\mathbf{v}$ is constant by \cite{KNSS}. If $\mathbf{v}$ is axi-symmetric,  then one naturally has the boundedness of $\Gamma$ for the ancient solution $\mathbf{v}$, since $\Gamma$ is scale invariant under the above natural scaling of Navier-Stokes equations. This explains that the boundedness of $\Gamma$ in Theorem \ref{thzhouqi} is acceptable.

The Strategy of proving Theorem \ref{thzhouqi} is as follows. The Liouville theorem for axi-symmetric Navier-Stokes equations without swirl (i.e. $v^\theta = 0$) has been established in \cite{KNSS}. Here, we are going to show that $\Gamma \equiv 0$ under the constraints of Theorem \ref{thzhouqi}. Our strategy is to apply the Nash-Moser method \cite{N,M1} to \eqref{Gamma/vtheta}. The main obstacle is the convection term $\mathbf{b} \cdot \nabla \Gamma$. In general, a kind of critical assumptions on $\mathbf{b}$ are necessary (for instance, $|\mathbf{b}| \le \frac{C_\star}{r}$, or $|\mathbf{b}| \le \frac{C_\star}{\sqrt{-t}}$ or $\mathbf{b}\in BMO^{-1}$. See \cite{KNSS,LZ11} for references). Here we use a modified Nash-Moser approach which is of independent interest, and exploit the inherent oscillation information on the stream function from $z$-periodicity. This overcomes the lack of crucial critical assumptions on $\mathbf{b}$.

The proof will be given in three steps. In Section 2, we use adapted Moser iteration to prove a mean value inequality which gives local maximum estimates. In Section 3, we use a Nash's inequality to provide lower bounds of certain solutions to \eqref{Gamma/vtheta}. In Section 4, we prove Harnack type estimates and finish the proof.

\section{Mean Value Inequality} \label{step1}
Usually, the first step of the Nash-Moser approach is to obtain a kind of mean value inequality like
$$\sup_{Q_{R/2}} |\Lambda | \le C \left\{ \frac{1}{R^5} \iint_{Q_R} |\Lambda|^2 dx ds \right\}^{\frac12},$$
for subsolutions $\Lambda$ to \eqref{Gamma/vtheta}, where $Q_R=  (-R^2,0) \times B_R $ is the standard parabolic cube, see for instance \cite{M1,LZ11,NU,SSSZ}. Here we have to use a rather different choice of space-time domains, adapted to the periodicity condition. Suppose $\mathbf{v}$ in Theorem \ref{thzhouqi} has period $Z_0$ in the $z$ direction. We write
\[
D_R = \{x \in \mathbb{R}^3 \, |\, 0\le r <R, \theta \in \mathbb{S}^1, 0\le z<Z_0\}
\]
for $R>0$, and $P_R= (-R^2,0) \times D_R$.

\begin{lemma}\label{lemma-21}
Assume that $\Lambda\ge 0$ is a Lipschitz subsolution  to \eqref{Gamma/vtheta} in $(-\infty,0] \times \mathbb{R}^3 $, with $\mathbf{b}$ as in \eqref{eq-b} and bounded, i.e. $\Lambda$ satisfies
\begin{equation} \label{eq-Phi}
\partial_t \Lambda - \Delta \Lambda + \frac{2}{r} \partial_r \Lambda + \mathbf{b} \cdot \nabla \Lambda\le0,
\end{equation}
in the sense of distributions. Also assume that $\Lambda$ has period $Z_0$ in the $z$-direction and $\Lambda\big|_{r=0}=0$. Then for any $R\ge1$, we have
\begin{equation} \label{mean-value-ineq}
\sup_{P_{R/2}} |\Lambda | \le C_\star \left\{ \frac{1}{R^4} \iint_{P_R} |\Lambda|^2 dx ds \right\}^{\frac12},
\end{equation}
where the constant $C_\star$ does not depend on $R$.
\end{lemma}

We will apply Lemma \ref{lemma-21} to the case $\Lambda = (\delta - \Phi)_+$, where $0 < \delta < 1$ is a constant and $\Phi$ will be defined in Section 4. In the proof of Lemma \ref{lemma-21} below, a key point is to use two-dimensional cut-off functions. It is worth noticing that the mean value inequality \eqref{mean-value-ineq} may not hold when $R$ approaches 0, while we only need it when $R \ge 1$.

\begin{proof}
We apply a modified version of Moser's iteration technique. In the proof, $C$ represents constants independent of $R$, whose value may change from line to line.  Set $\frac12 \le \sigma_2 < \sigma_1 \le 1$ and choose $\psi_1(s,r,\theta,z)=\eta_1(s)\phi_1(r)$ to be a smooth cut-off function defined on $P_1$ satisfying:
\begin{equation} \label{cutoff-phieta}
\begin{cases}
\text{supp}\,\phi_1 \subset D_{\sigma_1}, \quad \phi_1=1\,\, \text{on}\,\, D_{\sigma_2}, \\
\text{supp}\,\eta_1 \subset (-(\sigma_1)^2,0], \quad \eta_1(s)=1 \,\, \text{on}\,\, (-(\sigma_2)^2,0],\\
0 \le \phi_1 \le 1,\quad 0 \le \eta_1\le 1, \\
|\eta_1'| \lesssim \frac{1}{(\sigma_1-\sigma_2)^2}, \quad |\nabla \phi_1| \lesssim \frac{1}{\sigma_1-\sigma_2}.
\end{cases}
\end{equation}
Consider the cut-off functions $\psi_R(s,x) = \eta_1(\frac{s}{R^2})\phi_1(\frac{x}{R})$. Testing \eqref{eq-Phi} by $\Lambda \psi_R^2$ gives
\begin{align*}
& \quad -\frac{1}{2} \int_{-\infty}^t\int_{D_R} \left( \partial_s \Lambda^2 + (\mathbf{b}\cdot \nabla) \Lambda^2 + \frac{2}{r}\partial_r \Lambda^2 \right) \psi_R^2 dx ds\\
&\ge -\int_{-\infty}^t\int_{D_R} (\Delta \Lambda) \Lambda \psi_R^2 dx ds \\
&= \int_{-\infty}^t\int_{D_R} \left(|\nabla \Lambda|^2 \psi_R^2 + \Lambda \nabla \Lambda \cdot \nabla \psi_R^2 \right) dx ds,
\end{align*}
for any $t\le 0$. From now on we often abbreviate $\int_{-\infty}^t\int_{D_R}$ as $\iint$ and omit $dx ds$, unless there is any confusion. Since
\begin{align*}
\iint |\nabla \Lambda|^2 \psi_R^2\ge \iint \left( \frac12 |\nabla(\Lambda \psi_R)|^2 - \Lambda^2 |\nabla \psi_R|^2 \right),
\end{align*}
we get
\begin{align} \label{eq-21a}
&\quad \int (\Lambda^2 \psi_R^2)(t,\cdot) dx + \iint |\nabla (\Lambda \psi_R)|^2   \nonumber \\
&\le \iint -\left(  (\mathbf{b}\cdot \nabla) \Lambda^2 + \frac{2}{r}\partial_r \Lambda^2 \right) \psi_R^2  \\
&+ \iint \left(\Lambda^2 \partial_s \psi_R^2+2 \Lambda^2 |\nabla \psi_R|^2 -2 \Lambda \nabla\Lambda \cdot \nabla \psi_R^2 \right). \nonumber
\end{align}

Now we treat the right hand side term by term. For the first term, we use $\nabla \cdot \mathbf{b} =0$ to get
\begin{align} \label{eq-21b}
-&\iint (\mathbf{b}\cdot \nabla) \Lambda^2 \psi_R^2 = -\iint (v^r \partial_r \Lambda^2 +v^z \partial_z \Lambda^2) \psi_R^2 \nonumber \\
&=\iint (\partial_r v^r+ \frac{v^r}{r}+\partial_z v^z) \Lambda^2\psi_R^2 + v^r \Lambda^2 \partial_r \psi_R^2\nonumber\\
&= \iint v^r \Lambda^2 \partial_r \psi_R^2.
\end{align}
Define the usual angular stream function $L_\theta(t,r,z)$ by
\begin{align*}
\nabla \times (L_\theta \mathbf{e}_\theta) = v^r \mathbf{e}_r + v^z \mathbf{e}_z,
\end{align*}
Such an $L_\theta$ exists since $\mathbf{b}$ is divergence free. Moreover $L_\theta$ is periodic in $z$ with period $Z_0$ under our assumptions. We write
\begin{equation} \label{L-infty-(-1)}
v^r=-\partial_z L_\theta= -\partial_z (L_\theta(t,r,z)-L_\theta(t,r,0)).\\
\end{equation}
One can check that the oscillation of $L_\theta$ in $z$ satisfies
\begin{equation}\label{oscillationL}
  |L_\theta(t,r,z)-L_\theta(t,r,0)| \le \sup|v^r(t,r,\cdot)| Z_0 \lesssim 1,
\end{equation}
for any $z \in \mathbb{R}$. Hence we have
\begin{align} \label{eq-21c}
&\iint v^r \Lambda^2 \partial_r \psi_R^2= \iint \left(L_\theta(t,r,z)-L_\theta(t,r,0)\right) \partial_z \Lambda^2 \partial_r \psi_R^2\nonumber \\
&\quad \quad\le C \iint \Lambda^2 (\partial_r \psi_R)^2 + \frac{1}{8}\iint (\partial_z \Lambda)^2 \psi_R^2\nonumber \\
&\quad \quad\le \frac{C}{(\sigma_2-\sigma_1)^2R^2} \iint_{P_{\sigma_1 R}} \Lambda^2  + \frac{1}{8}\iint (\partial_z (\Lambda \psi_R))^2.
\end{align}
For the second term in \eqref{eq-21a}, using $\Lambda\big|_{r=0}=0$ we get
\begin{align}\label{eq-21d}
-\iint \frac{2}{r} \partial_r \Lambda^2 \psi_R^2 &= \iint 2\Lambda^2  \frac{\partial_r\psi_R^2}{r}\nonumber \\
&\le \frac{C}{(\sigma_2-\sigma_1)R^2}\iint_{P_{\sigma_1 R}} \Lambda^2 .
\end{align}
The last three terms in \eqref{eq-21a} are easier:
\begin{align} \label{eq-21e}
&\quad \iint \Lambda^2 \partial_s \psi_R^2 + 2 \Lambda^2 |\nabla \psi_R|^2 - 2\Lambda \nabla\Lambda \cdot \nabla \psi_R^2 \nonumber\\
&\le \frac{C}{(\sigma_2-\sigma_1)^2R^2} \iint_{P_{\sigma_1 R}} \Lambda^2  + \frac{1}{8} \iint |\nabla (\Lambda \psi_R)|^2.
\end{align}
Combing \eqref{eq-21a},\eqref{eq-21b},\eqref{eq-21c},\eqref{eq-21d},\eqref{eq-21e}, and using the properties of the cutoff functions \eqref{cutoff-phieta}, we arrive at
\begin{align}
\label{energy}
&\sup_{t\le 0} \|(\Lambda\phi_R)(t,\cdot)\|_{L_x^2(D_{R})}^2 + \|\nabla (\Lambda\psi_R)\|_{L_t^2 L_x^2(P_{R})}^2 \nonumber\\
&\quad \quad \le \frac{C}{(\sigma_1-\sigma_2)^2R^2} \iint_{P_{\sigma_1 R}} \Lambda^2 .
\end{align}

We have to use the following Sobolev embedding inequality for periodic functions:
\begin{align} \label{sobolev-embedding}
\|f\|_{L_x^3(D_{1})} \le C\|\nabla f\|_{L_x^2(D_{1})},
\end{align}
for any $f$ having period $Z_0$ in $z$ and compactly supported in $r \le 1$ in the other two dimensions. To verify \eqref{sobolev-embedding}, one can argue as follows. Choose a cut-off function \begin{equation*}
g(z) = \begin{cases}
1, \quad 0<z\le NZ_0,\\
2-\frac{z}{NZ_0}, \quad NZ_0 \le z < 2NZ_0,\\
0, \quad \text{otherwise},
\end{cases}
\end{equation*}
with $N$ large. By the usual Sobolev embedding, after extending $f$ to the whole space in the periodic way along the $z$ axis, we deduce
\begin{align*}
&\quad \frac{N^\frac13}{2}\|f\|_{L^3(D_1)} \le \|fg\|_{L^3(R^3)} \\
& \le C N^{\frac16} \|fg\|_{L^6(R^3)} \le C N^{\frac16}\|\nabla (fg)\|_{L^2(R^3)}\\
&\le C N^{\frac16}\| (\nabla f)g\|_{L^2(R^3)} + CN^{\frac16} \|f(\partial_z g)\|_{L^2(R^3)}\\
&\le C N^\frac23\| \nabla f\|_{L^2(D_1)} + C N^{-\frac13} \|f\|_{L^2(D_1)},
\end{align*}
which clearly implies \eqref{sobolev-embedding}. By scaling argument in the $x_1$ and $x_2$ directions, we have
\begin{align*}
R^{-\frac{2}{3}} \|(\Lambda\psi_R)(t,\cdot)\|_{L_x^3(D_{R})} \le C\|\nabla (\Lambda\psi_R)\|_{L_x^2(D_{R})}.
\end{align*}
We emphasize here that $R$ should be bounded from below, say by $1$. Interpolation from \eqref{energy} gives
$$\left(\frac{1}{R^4} \iint_{P_{\sigma_2 R}} \Lambda^{\frac{5}{2}} \right)^\frac{2}{5} \le \frac{C}{\sigma_1-\sigma_2} \left(\frac{1}{R^4} \iint_{P_{\sigma_1 R}} \Lambda^2\right)^\frac{1}{2},$$
where $C$ does not depend on $R$.
Observe that $\Lambda^{(\frac54)^k}, \quad k \ge 1$ are also positive subsolutions to \eqref{Gamma/vtheta}. Hence one can clearly repeat the above estimates to derive
\begin{align*}
&\quad\quad\left(\frac{1}{R^4} \iint_{P_{\sigma_{2k} R}} \Lambda^{2\times(\frac54)^{k+1}} \right)^\frac{2}{5}  \\
&\le \frac{C}{\sigma_{1k}-\sigma_{2k}} \left(\frac{1}{R^4} \iint_{P_{\sigma_{1k} R}} \Lambda^{2\times(\frac54)^k}\right)^\frac{1}{2},
\end{align*}
for any $\frac{1}{2}\le\sigma_{2k}<\sigma_{1k}\le 1$. This is equivalent to
\begin{align*}
&\quad\quad\left(\frac{1}{R^4} \iint_{P_{\sigma_{2k} R}} \Lambda^{2\times(\frac54)^{k+1}} \right)^{\frac12\times (\frac45)^{k+1}} \\
&\le \left(\frac{C}{\sigma_{1k}-\sigma_{2k}}\right)^{(\frac45)^k} \left(\frac{1}{R^4} \iint_{P_{\sigma_{1k} R}} \Lambda^{2\times(\frac54)^k}\right)^{\frac12\times (\frac45)^k}.
\end{align*}
It remains to choose $\sigma_{1k}$ and $\sigma_{2k}$ converging to $\frac{1}{2}$ and iterate the above inequalities. This process is standard \cite{M1}, thus omitted.
\end{proof}

\section{Estimates for $-\ln \Phi$}
In this section, we prove two important lemmas, which will lead to lower bounds of certain solutions of \eqref{Gamma/vtheta} in Section 4. The proofs are based on ideas in \cite{LZ11}, \cite{CSTY1} and \cite{Z}. The observations \eqref{L-infty-(-1)}, \eqref{oscillationL} made in Section 2 will be essentially used here again.

Assume that $\Phi$ is a positive $z$-periodic solution to \eqref{Gamma/vtheta} in $P_R=(-R^2,0) \times D_R$. Without loss of generality, we let the $z$-period to be $Z_0=1$ from now on, for simplicity of presentation. We also assume that $\Phi\big|_{r=0} \ge \frac12$. In Section 4, $\Phi$ will be taken as \eqref{def-Phi}.

We denote $\Psi= -\ln \Phi$. The equation for $\Psi$ reads
\begin{equation} \label{eq-Psi}
\partial_t \Psi + \mathbf{b} \cdot \nabla \Psi + \frac{2}{r} \partial_r \Psi - \Delta \Psi + |\nabla \Psi|^2=0.
\end{equation}
Choose cut-off functions $\zeta_R(r,\theta,z)=\zeta_1(\frac{r}{R})$ such that
\begin{equation} \label{cutoff-zeta}
\begin{cases}
\zeta_R=1,\quad \text{for}\,\, x\in D_{R/2},\\
|\partial_r \zeta_R| \lesssim \frac{1}{R},\,\, \partial_\theta \zeta_R=\partial_z \zeta_R=0.
\end{cases}
\end{equation}
By multiplying \eqref{eq-Psi} with $\zeta_R^2$ and integrating in the space variables only, we get
\begin{align*}
&\quad\partial_t \int_{D_R} \Psi \zeta_R^2 dx + \int_{D_R} |\nabla \Psi|^2 \zeta_R^2 dx \\
&= \int_{D_R} -\mathbf{b} \cdot \nabla \Psi \zeta_R^2 - \frac{2}{r} \partial_r \Psi \zeta_R^2 - \nabla \Psi \cdot \nabla \zeta_R^2 \\
&\le \int_{D_R} -\mathbf{b} \cdot \nabla \Psi \zeta_R^2 - \frac{2}{r} \partial_r \Psi \zeta_R^2 + \frac{1}{6}|\nabla \Psi|^2 \zeta_R^2 + C |\nabla \zeta_R|^2  \nonumber.
\end{align*}
Using the properties \eqref{cutoff-zeta}, we arrive at
\begin{align}\label{eq-3a}
&\quad \partial_t \int \Psi \zeta_R^2 dx + \frac{5}{6}\int |\nabla \Psi|^2 \zeta_R^2 dx \nonumber\\
& \le C + \int \left(-\mathbf{b} \cdot \nabla \Psi - \frac{2}{r} \partial_r \Psi \right)\zeta_R^2 dx.
\end{align}
The drift term can be estimated in the spirit of \eqref{eq-21b} and \eqref{eq-21c},
\begin{align}\label{eq-3b}
\int -\mathbf{b} \cdot \nabla \Psi \zeta_R^2 &= \int (v^r \partial_r \zeta_R^2+ v^z \partial_z \zeta_R^2) \Psi\nonumber \\
&=-\int \partial_z (L_\theta(r,z,t)-L_\theta(r,0,t)) \partial_r \zeta_R^2 \Psi\nonumber\\
&=\int (L_\theta(r,z,t)-L_\theta(r,0,t)) \partial_r \zeta_R^2 \partial_z \Psi\nonumber\\
&\le C+\frac{1}{6} \int |\nabla \Psi|^2 \zeta_R^2.
\end{align}  Here we just used $| D_R | \sim R^2$ for large $R$.

To proceed, we need the weighted Poincar\'e inequality in our periodic domain $D_R\,(R\ge 1)$,
\begin{align} \label{weighted-poincare}
\int_{D_R} |\Psi-\bar{\Psi}|^2\zeta_R^2 dx \le C R^2 \int_{D_R} |\nabla \Psi|^2 \zeta_R^2 dx,
\end{align}
where
\[
\bar{\Psi} = \left( \int \zeta_R^2 dx \right)^{-1} \int \Psi\zeta_R^2 dx.
\]To check this we first use the usual weighted Poincar\'e inequality in two dimensions to deduce
\begin{align*}
\int_{D_R} |\Psi-[\Psi](z)|^2\zeta_R^2 dx \le C R^2 \int_{D_R} |\nabla \Psi|^2 \zeta_R^2 dx,
\end{align*}
where
\[
[\Psi](z)=\left( \iint \zeta_R^2 rdr d\theta \right)^{-1}\iint \Psi \zeta_R^2 rdr d\theta.
\] Moreover, since $[\Psi]$ depends only on $z$, and $\bar{\Psi} = Z^{-1}_0 \int^{Z_0}_0 [\Psi](z) dz$, we have
\begin{align*}
\int_{D_R} |[\Psi]-\bar{\Psi}|^2 \zeta_R^2 dx &\le C R^2 \int_0^{Z_0} |[\Psi]-\bar{\Psi}|^2 dz\\
 &\le C R^2 \left(\int_0^{Z_0} |\partial_z [\Psi]| dz \right)^2 \\
&\le \frac{C}{R^2} \left(\int_{D_R} |\partial_z \Psi| \zeta_R^2 dx\right)^2 \\
&\le C \int_{D_R} |\partial_z \Psi|^2 \zeta_R^2 dx.
\end{align*} Here we have used a one-dimensional Sobolev imbedding, passing from line 1 to line 2.
This proves \eqref{weighted-poincare}.

Now integration by parts and \eqref{weighted-poincare} give
\begin{align} \label{eq-3c}
-\int \frac{2}{r} \partial_r \Psi \zeta_R^2 dx &= 2\iint (\Psi-\bar{\Psi}) \zeta_R^2 d\theta dz \big|_{r=0}  \nonumber \\
&\quad+ 2 \int (\Psi-\bar{\Psi}) \frac{\partial_r \zeta_R^2}{r} dx\nonumber\\
&= C-C \bar{\Psi} +2\int (\Psi-\bar{\Psi}) \frac{\partial_r \zeta_R^2}{r} dx\nonumber\\
&\le C-C \bar{\Psi}+ \frac{1}{6} \int |\nabla \Psi|^2 \zeta_R^2 dx \nonumber\\
&\quad + C R^2 \int \left(\frac{\partial_r \zeta_R}{r}\right)^2 dx \nonumber\\
&\le C'-C \bar{\Psi}+ \frac{1}{6} \int |\nabla \Psi|^2 \zeta_R^2 dx.
\end{align}
Hence, from \eqref{eq-3a},\eqref{eq-3b},\eqref{eq-3c},  we get a crucial differential inequality:
\begin{align} \label{inequality-Psi}
\partial_t \int \Psi \zeta_R^2 dx + C_1 \bar{\Psi} \le -\frac{1}{2} \int |\nabla \Psi|^2 \zeta_R^2 dx + C_2,
\end{align}
for $t\in [-R^2,0]$ and $C_1,\, C_2>0$ independent of $R$. At this point, we claim that the following lemma holds, since the same arguments in \cite{LZ11} can be applied to our situation with some adjustments on the region of integration.

\begin{lemma} \label{lemma-31}
Let $\Phi \le 1$ be a positive z-periodic solution to \eqref{Gamma/vtheta} in $P_R ( R\ge 1)$ which satisfies
\begin{align} \label{lower-bound-l1}
\|\Phi\|_{L^1(P(\frac{R}{2}))}\ge \kappa R^4,
\end{align}
for some $0<\kappa<1$. Moreover we assume that $\Phi\big|_{r=0}\ge \frac{1}{2}$. Then there holds
\begin{align} \label{lemma-31-conclusion}
-\int \zeta_R^2(x) \ln \Phi(t,x) dx \le M R^2,
\end{align}
for all $t\in [-\frac{\kappa R^2}{4},0]$ and some positive constant $M$ depending only on $\kappa$.
\end{lemma}

\begin{proof}

For the sake of completeness, we present the proof here. Note that
\[
d\mu=\frac{1}{R^2} \Big(\int \zeta_1^2 dx\Big)^{-1} \zeta_R^2 dx
\] is a probability measure.
By Nash's inequality(see Lemma \ref{lemma-nash} below) and the weighted Poincar\'e inequality \eqref{weighted-poincare}, and since $\Psi = - \ln \Phi$,
\begin{align}\label{eq-31a}
&\quad\left|\ln\left(\int_{D_R}\Phi d\mu\right) + \int_{D_R} \Psi d\mu \right|^2 \left(\int_{D_R} \Phi d\mu\right)^2\nonumber\\
&\le |\sup \Phi|^2  \int_{D_R} \left|\Psi -\bar{\Psi}\right|^2 d\mu\nonumber\\
&\le C_3 \int_{D_R} |\nabla \Psi|^2 \zeta_R^2 dx.
\end{align}
For simplicity we write $a=\int \zeta_1^2 dx>0$. Plugging \eqref{eq-31a} into \eqref{inequality-Psi} gives
\begin{align}\label{eq-31b}
&\quad aR^2 \partial_t \bar{\Psi}(t) + C_1 \bar{\Psi}(t) \nonumber \\
&\le C_2 - \frac{1}{2C_3} \left|\ln\int_{D_R}\Phi d\mu + \int_{D_R} \Psi d\mu \right|^2 \left(\int_{D_R} \Phi d\mu\right)^2.
\end{align}
Now we consider the set
\begin{align*}
W=\{s\in [-\frac14 R^2,0]:\int_{D_\frac{R}{2}} \Phi(s) dx \ge \frac{\kappa}{2} R^2\},
\end{align*}
and denote its characteristic function by $\chi(s)$. Due to the condition \eqref{lower-bound-l1}, we have
\begin{align*}
\kappa R^4 \le \int_{P_{R/2}} \Phi dx dt &= \int_W \int_{D_{R/2}} \Phi(s) dx ds \\
& \quad + \int_{[-R^2/4,0]- W} \int_{D_{R/2}} \Phi(s) dx ds \\
&\le |W||D_{R/2}| \sup_{D_{R/2}}|\Phi|  + \frac{R^2}{4} \frac{\kappa R^2}{2}\\
&\le \frac{\pi R^2}{4} |W| + \frac{\kappa R^4}{8}.
\end{align*}
This gives
\begin{align} \label{lower-bound-|W|}
|W| \ge \frac{\kappa R^2}{2}.
\end{align}
From $aR^2 \partial_t \bar{\Psi} + C_1\bar{\Psi} \le C_2$, it is easy to derive that
\begin{align} \label{eq-31c}
\bar{\Psi}(s_2) \le \bar{\Psi}(s_1) + \frac{C_2}{C_1},
\end{align}
for any $-\frac{R^2}{4}\le s_1\le s_2 \le 0$. If for some $\frac{-R^2}{4} \le s\le \frac{-\kappa R^2}{4}$, there holds
$$\bar{\Psi}(s) \le 2\left|\ln \frac{\kappa}{2a}\right|+\frac{8a\sqrt{C_2C_3}}{\kappa},$$
then due to \eqref{eq-31c}, the conclusion \eqref{lemma-31-conclusion} holds with
\begin{equation*}
M=a\left(2\left|\ln \frac{\kappa}{2a}\right|+\frac{8a\sqrt{C_2C_3}}{\kappa} + \frac{C_2}{C_1}\right).
\end{equation*}
Otherwise, for all $-\frac{R^2}{4}\le s_1\le s_2 \le 0$ we have
$$\bar{\Psi}(s) \ge 2\left|\ln \frac{\kappa}{2a}\right| +\frac{8a\sqrt{C_2C_3}}{\kappa}.$$
For $s\in W\cap [-\frac{R^2}{4},-\frac{\kappa R^2}{4}]$, one has
$$\ln \int_{D_R} \Phi(s) d\mu \ge \ln \int_{D_{R/2}} \Phi(s) d\mu \ge \ln\frac{\kappa}{2a}.$$
In this case, \eqref{eq-31b} and $\Psi \ge 0$ gives
\begin{equation} \label{eq-31d}
aR^2 \partial_t \bar{\Psi}(t) \le a R^2 \partial_t \bar{\Psi}(t) + C_1 \bar{\Psi}(t) \le - C_4 \chi(t) \bar{\Psi}(t)^2,
\end{equation}
for $t \in [-\frac{R^2}{4},-\kappa\frac{R^2}{4}]$. Note that \eqref{lower-bound-|W|} implies
\begin{align*}
\int_{-R^2/4}^{-\kappa R^2/4} \chi(s) ds \ge \frac{\kappa R^2}{4}.
\end{align*}
Solving the Riccati type inequality \eqref{eq-31d} clearly gives an absolute upper bound
for $\bar{\Psi}(-\frac{\kappa R^2}{4})$. See \cite{LZ11} Lemma 3.2 for details. The conclusion \eqref{lemma-31-conclusion} follows immediately by \eqref{eq-31c}.
\end{proof}

The Nash inequality used earlier can be found in \cite{CSTY2}. We give an easier proof here.
\begin{lemma}\label{lemma-nash}
Let $\mu$ be a probability measure. Then for any integrable function $\Phi >0$ we have
\begin{align} \label{nash-inequality}
& \left| \ln \left( \int \Phi d\mu \right) - \int \ln \Phi d\mu \right| \left(\int \Phi d\mu\right) \nonumber\\
&\quad\le (\sup \Phi) \int \left|\ln \Phi-\int \ln \Phi d\mu \right| d\mu.
\end{align}
\end{lemma}
\begin{proof}
After multiplying $\Phi$ by a constant which leaves \eqref{nash-inequality} invariant, one may assume that $\int \ln \Phi d\mu=0$. In this case, Jensen's inequality gives
\begin{equation*}
\ln \int \Phi d\mu \ge \int \ln \Phi d\mu = 0.
\end{equation*}
For the convex function $f(\alpha)=\alpha\ln \alpha$, using Jensen's inequality again we get
\begin{align*}
 \ln  \left(\int \Phi d\mu \right)  \left(\int \Phi d\mu\right) \le \int \Phi \ln \Phi d\mu \\
 \le (\sup \Phi) \int |\ln \Phi| d\mu.
\end{align*}
This proves \eqref{nash-inequality}.
\end{proof}

We shall need another auxiliary lemma giving a lower bound for $\int_{P_R} \Phi\, dx dt$, which makes Lemma \ref{lemma-31} applicable.

\begin{lemma}\label{lemma-32}
Let $\Phi$ be a nonnegative z-periodic solution(with period $Z_0=1$ in the $z$ direction) to \eqref{Gamma/vtheta} in $P_R\,\, (R\ge 1)$, satisfying
$$\Phi\big|_{r=0} \ge \frac{1}{2}.$$
Then
\begin{equation}\label{lemma-32-conclusion}
\|\Phi \|_{L^1(P_R)} \ge \kappa R^4,
\end{equation}
for some absolute constant $\kappa > 0$ independent of R.
\end{lemma}
\begin{proof}
We may assume that $\Phi > 0$. Otherwise one can work with $\Phi+\epsilon$ and let $\epsilon \searrow 0$.

Consider cut-off functions $\psi=\psi_R(t,x)$ compactly supported on $P_R$, satisfying
\begin{equation} \label{cutoff-psi}
\begin{cases}
\psi_R = 1, \quad \text{for}\quad (t,x) \in [-\frac34 R^2, -\frac14 R^2] \times D_{R/2}, \\
\partial_z \psi_R =0 , \quad |\nabla \psi_R| \lesssim \frac{1}{R},\\
|\partial_t \psi_R|,\,\, |\nabla^2 \psi_R| \lesssim \frac{1}{R^2}.
\end{cases}
\end{equation} For simplicity of presentation, we will drop the index $R$ in $\psi_R$ unless stated otherwise.
Let us test \eqref{Gamma/vtheta} by $\frac{1}{2\sqrt{\Phi}}\psi_R^2$ in the domain $P_R$:
\begin{align}\label{eq-32a}
\int_{P_R}  -\sqrt{\Phi} \partial_t \psi^2 &+ \frac{2}{r} \partial_r(\sqrt{\Phi}) \psi^2 + \mathbf{b} \cdot \nabla\sqrt{\Phi} \psi^2= \int_{P_R}  \Delta \Phi \frac{\psi^2}{2\sqrt{\Phi}}\nonumber\\
&= \int_{P_R}  \sqrt{\Phi} \Delta \psi^2 + 4\int_{P_R}  |\nabla (\Phi^\frac14)|^2 \psi^2.
\end{align}
The singular drift term can be estimated similarly as before,
\begin{align} \label{eq-32b}
\int \frac{2}{r} \partial_r \sqrt{\Phi} \psi^2 &= -2 \iiint \sqrt{\Phi}\big|_{r=0} \psi^2 d\theta dz dt - 2 \int  \sqrt{\Phi} \frac{\partial_r \psi^2}{r}\nonumber\\
&\le - \kappa_1 R^2 -2\int \sqrt{\Phi} \frac{\partial_r \psi^2}{r},
\end{align}
where $\kappa_1$ is a positive constant. Then we again use $\nabla \cdot \mathbf{b} =0$ and $v^r = - \partial_z(L_\theta-L_\theta(t,r,0))$ to get
\begin{align} \label{eq-32c}
\int \mathbf{b} \cdot \nabla\sqrt{\Phi} \psi^2 &= -\int v^r \sqrt{\Phi} \partial_r \psi^2 \nonumber\\
&= \int (L_\theta-L_\theta(t,r,0)) \partial_z \sqrt{\Phi} \partial_r \psi^2 \nonumber\\
&\le \int |\nabla (\Phi^\frac14)|^2 \psi^2 + C\int \sqrt{\Phi} (\partial_r \psi)^2.
\end{align}
We plug \eqref{eq-32b}, \eqref{eq-32c} into \eqref{eq-32a} to get
\begin{align*}
\int_{P_R} \sqrt{\Phi} (-\partial_t \psi^2 -2\frac{\partial_r \psi^2}{r} + C (\partial_r \psi)^2 -\Delta \psi^2) \ge \kappa_1 R^2.
\end{align*}
Due to \eqref{cutoff-psi} we have
\begin{align*}
\int_{P_R} \sqrt{\Phi} \frac{1}{R^2} \ge \kappa_2 R^2,
\end{align*}
for some positive constant $\kappa_2$ independent of $R$. It remains to conclude \eqref{lemma-32-conclusion} using H\"older's inequality.
\end{proof}

\section{Harnack estimates}

Let us work with $|\Gamma|\le 1$ and $Z_0=1$. It suffices to prove that $\Gamma \equiv 0$ to deduce Theorem \ref{thzhouqi}, as explained in the strategy of proof.

Let $R>0$. We may assume that
$$\sup_{P_R} \Gamma \le -\inf_{P_R} \Gamma.$$
Otherwise consider $-\Gamma$. Let
\begin{equation} \label{def-Phi}
\Phi = \frac{\Gamma - \inf_{P_R} \Gamma}{\sup_{P_R} \Gamma - \inf_{P_R} \Gamma}.
\end{equation}
Then $0\le \Phi \le 1$ and $\Phi\big|_{r=0}\ge \frac12$. By Lemma \ref{lemma-31} and Lemma \ref{lemma-32} we deduce that for all $t\in [-\frac{\kappa R^2}{4},0],$
$$-\int \zeta_R^2(x) \ln \Phi(t,x) dx \le M R^2.$$
By Chebyshev's inequality, for any $0<\delta<1$ and $t\in [-\frac{\kappa R^2}{4},0]$,
\begin{equation} \label{eq-11a}
|\{x\in D_{R/2}: \Phi(t,x)\le \delta\}| \le \frac{M R^2}{|\ln \delta|}.
\end{equation}
Since $(\delta - \Phi)_+$ is a nonnegative Lipschitz subsolution to \eqref{Gamma/vtheta}, we apply Lemma \ref{lemma-21} and use \eqref{eq-11a} to deduce
\begin{align*}
\sup_{P_{\sqrt{\kappa}R/4}} (\delta-\Phi)_+ &\lesssim \left\{\frac{1}{R^4} \iint_{P_{\sqrt{\kappa}R/2}} (\delta-\Phi)_+^2 dx dt\right\}^\frac{1}{2} \\
&\le \left\{\frac{M\delta^2}{R^2|\ln \delta|}\right\}^\frac12.
\end{align*}
Choose a $\delta$ small enough we get a point-wise lower bound
$$\Phi(t,x) \ge \frac{\delta}{2},$$
for $(t,x) \in P_{\sqrt{\kappa}R/4}$. This implies
$$\left(\sup_{P_{\sqrt{\kappa}R/4}}- \inf_{P_{\sqrt{\kappa}R/4}}\right) \Phi \le 1-\sigma,$$
for some constant $\sigma >0$. Hence
\begin{equation} \label{eq-11b}
\left(\sup_{P_{\sqrt{\kappa}R/4}}- \inf_{P_{\sqrt{\kappa}R/4}}\right) \Gamma \le (1-\sigma)\left(\sup_{P_{R}}- \inf_{P_{R}}\right) \Gamma.
\end{equation}
Iterating \eqref{eq-11b} for a sequence of $R_k \to \infty$, we get $\Gamma \equiv \Gamma(t=0, x=0)=0$. As mentioned earlier, this implies $\mathbf{v} = c \mathbf{e}_z$ with $c$ being a constant.

\section*{Acknowledgement}
Z.L. was supported by NSFC (grant No. 11725102) and National Support Program for Young Top-Notch Talents. Q.S.Z. wishes to thank the Simons Foundation for its support, and Fudan University for its hospitality during his visit.

\bibliographystyle{plain}
\bibliography{LRZ-bib}

\end{document}